\documentclass[a4paper,11pt,reqno,pdftex]{amsart}
\usepackage{amsfonts,amsmath,amssymb,amsthm}
\usepackage{graphicx}
\usepackage{fixmath}
\usepackage{hyperref}
\usepackage{url}
\usepackage{a4wide}

\newtheorem{theorem}{Theorem}[section]

\newtheorem{remark}{Remark}

\newcommand{\ds}{\displaystyle}

\newcommand{\ol}{\overline}

\newcommand{\ra}{\rightarrow}

\newcommand{\Card}{\mathop{\textsf{Card}}\nolimits}
\newcommand{\pol}{\mathop{\textsf{pol}}\nolimits}

\newcommand{\Tr}{\mathcal{T}\!\mathit{r}}

\title{Coupon collecting and transversals of hypergraphs}

\author{Marcel Wild}
\address{Marcel Wild\\
University of Stellenbosch\\
South Africa
}
\email{mwild@sun.ac.za}

\author{Svante Janson}
\address{Svante Janson\\
University of Uppsala\\
Sweden
}
\email{svante@math.uu.se}

\author{Stephan Wagner}
\address{Stephan Wagner\\
University of Stellenbosch\\
South Africa
}
\email{swagner@sun.ac.za}

\author{Dirk Laurie}
\address{Dirk Laurie\\
University of Stellenbosch\\
South Africa
}
\email{dpl@sun.ac.za}

\keywords{coupon collector, transversal}

\begin{document}

\maketitle

\begin{abstract}
The classic Coupon-Collector Problem (CCP) is generalized. Only basic probability theory is used. Centerpiece rather is an algorithm that efficiently counts all $k$-element transversals of a set system.
\end{abstract}

\section{Introduction} \label{sec:intro}

In the popular game of Roulette a small metal  bullet is spun and stopped at random on one of the $w=37$ numbers $0,1,2, \dots, 36$. Apart from $0$ each one of these numbers has several properties.
For example 13 is at the same time odd, black, in the 2nd 12,  and $1c$ (in the first column); see Figure~\ref{fig:roulette}. We will show how to compute the  expected
time to encounter, in successive draws at random, all these properties: even, odd, red, black, 1--18, 19--36, 1st 12, 2nd 12, 3rd 12, 1c, 2c, 3c. 

Our setting is as follows. Let $W$ be a set whose $w$ many elements $c$ will be viewed as ``coupons''. Let ${\mathcal G} = \{G_1, \cdots, G_h\}$ be any family of nonempty (not necessarily distinct) subsets. Thus $\bigcup {\mathcal G} \subseteq W$. By definition $G_i$ contains exactly the coupons $c$ of the $i$-th \emph{goal} (purpose, property, etc.) Put another way, each fixed coupon $c \in W$ is \emph{multipurpose} in the sense that it can serve many goals $i, j$, and so forth, according to the sets $G_i, G_j, \cdots$ that contain $c$. If $\bigcup {\mathcal G} = W$, then every coupon has at least one goal. It is convenient to imagine the $w$ many coupons as being located in an urn.

In a \emph{length $n$ trial} a set of $n$ coupons is picked at random one by one, and all occuring goals are recorded. For any picked coupon some of its goals may have occured already and are not again taken into account. In a trial \emph{with replacement} each coupon is put back into the urn after its goals have been ticked off. Thus at each moment every fixed coupon is drawn with probability $\frac{1}{w}$. In a trial \emph{without replacement} no drawn coupons are put back. Then necessarily $n \leq w$. Again at each moment every coupon remaining in the urn has the same probability to be drawn (some number $\geq \frac{1}{w}$). A trial is \emph{successful} if all $h$ goals show up. It is handy to call a successful trial \emph{sharply successful} if the $h$ goals are only completed in the \emph{last} draw. 

The \emph{Generalized Coupon-Collector Problem} (GCCP) is to calculate the expected length $\ell$ of a sharply successful trial.  We shall use the notation $\ell = \ell_r ({\mathcal G})$ and $\ell = \ell_{nr} ({\mathcal G})$ for a GCCP with \emph{replacement}, respectively \emph{without replacement}.

Coming back to Roulette, the coupons are the numbers $0,1, \dots, 36$ and they constitute the set $W$. As seen, all coupons have several properties, except for $0$
which has none (so $\bigcup {\mathcal G} \subset W$). Using the method of \S\ref{sec:with-rep} it turns out that the expected length of a sharply successful trial in this
GCCP with replacement is $$\ds\frac{54728027202913}{7600186994400} \approx 7.201.$$

If after each ``drawing'' one prevents the spinning wheel from delivering the same number again (thus having a GCCP {\it without} replacement) the corresponding number is obviously smaller; in fact it is 
$$\ds\frac{65774035502891}{10043104242600} \approx 6.549.$$ 
Notice that we can also model our multipurpose coupons $c_i \in W$ with {\it different} drawing probabilities $p_i$ as follows (for simplicity we only focus on drawings {\it with} replacement ). If without great loss of generality all $p_i$'s are rational, say $p_i = m_i/wM$, replace each $c_i$ by $m_i$ copies $c'_i, c''_i, \cdots$ which all have exactly the same goals as $c_i$. Let $W'$ be the new $wM$-element set of coupons and let ${\mathcal G}'$ match ${\mathcal G}$ in the obvious way. Then $\ell_{nr}({\mathcal G}')$ is the expected length of a sharply succesful trial with {\it original} coupons $c_i \in W$ if they were subject to the drawing probabilities $p_i$.

There is a simple situation where  the coupons in $W$ {\it already} furnish (but do not ``have'') potentially different drawing probabilities. Namely, suppose that each $c \in W$ has exactly {\it one} of $h$ goals which we then refer to as its {\it type}, and that different coupons can have the same type. Then ${\mathcal G}^\ast = \{G_1, \cdots, G_h\}$ is a {\it partition}\footnote{We henceforth write ${\mathcal G}^\ast$, not ${\mathcal G}$, in case of a partition.} of $W$ and $p_i = |G_i|/|W|$ is the probability for drawing a type $i$ coupon. This matches the ``classic'' Coupon-Collector Problem (CCP) except that in the latter framework there is no ${\mathcal G}^\ast$ but simply an unbounded supply of coupons. Each belongs to exactly one of $h$ types, the $i$-th type being drawn with probability $p_i$. The expected length $\ell (p_1, \cdots, p_h)$ of a sharply successful trial is known to be [David-Barton, p.269]

\begin{equation}
\ell (p_1, \ldots, p_h)  =  \ds\sum_{1 \leq i \leq h} \frac{1}{p_i} - \ds\sum_{1\leq i \leq j \leq h} \ \frac{1}{p_i+p_j} + \ds\sum_{1\leq i < j < k \leq h} \frac{1}{p_i + p_j + p_k}\ -\ \cdots \pm \frac{1}{p_1+\cdots  + p_h}. \label{eq:1}
\end{equation}

In particular, if $p_1 = p_2 = \cdots = p_h = \frac{1}{h}$ (call this the \emph{homogeneous} CCP) then (\ref{eq:1}) can be shown 
[Feller 1957, Example IX.3(d)] to simplify to  

\begin{equation}
\ell \left(\frac{1}{h}, \cdots,  \frac{1}{h}\right)   =  h H (h), 
\label{eq:2}
\end{equation}

where $H(h) := 1 + \frac{1}{2} + \cdots + \frac{1}{h}$ is the harmonic number.

For instance, setting $h=6$ in (\ref{eq:2}) one finds that a die has to be thrown 14.7 times on average until all numbers have shown up. The CCP has been studied by many authors, which is evident when feeding Google Scholar with ``coupon collecting''.  

Notwithstanding our sweeping generalization of the CCP (whose formula (1) is intimidating enough) the present article does not feature subtle probability arguments, but rather revolves around an efficient algorithm (the {\it transversal $e$-algorithm}) for counting the $k$-element transversals of a set system $(k = 1,2, \dots)$. The connection to coupon collecting, straightforward but unexploited so far, is discussed in \S\ref{sec:counting}. Surprisingly perhaps, our approach to the GCCP appeals more to the GCCP \emph{without replacement} (\S\ref{sec:without-rep}). Only afterwards in \S\ref{sec:with-rep} we tackle the GCCP with replacement. In \S\ref{sec:algorithm} details on the origin and workings of the transversal $e$-algorithm are provided. 

A numerical evaluation of our method pitted against the inclusion-exclusion approach (\ref{eq:1}), as well as applications to e.g.\ chess follow in \S\ref{sec:comparison}--\ref{sec:alpha}.

We shall use the notation $[h]: = \{1,2, \ldots, h\}$ for positive integers $h$.

%Applications of the CCP are plentiful, some of them outlined (with

%references) in \cite[\S 1]{BH}. For applications of urn models in general

%see \cite{JK}.  The GCCP framework seems to be new. We mention %that some

%aspect of it (\S\ref{subsec:alpha}) fits the FDP hat \cite[sec. 9.3]{FDP}

%in a clumsy way, yet the arising system of ordinary differential equations

%is usually not integrable. 

\begin{figure}
\begin{center}
\includegraphics{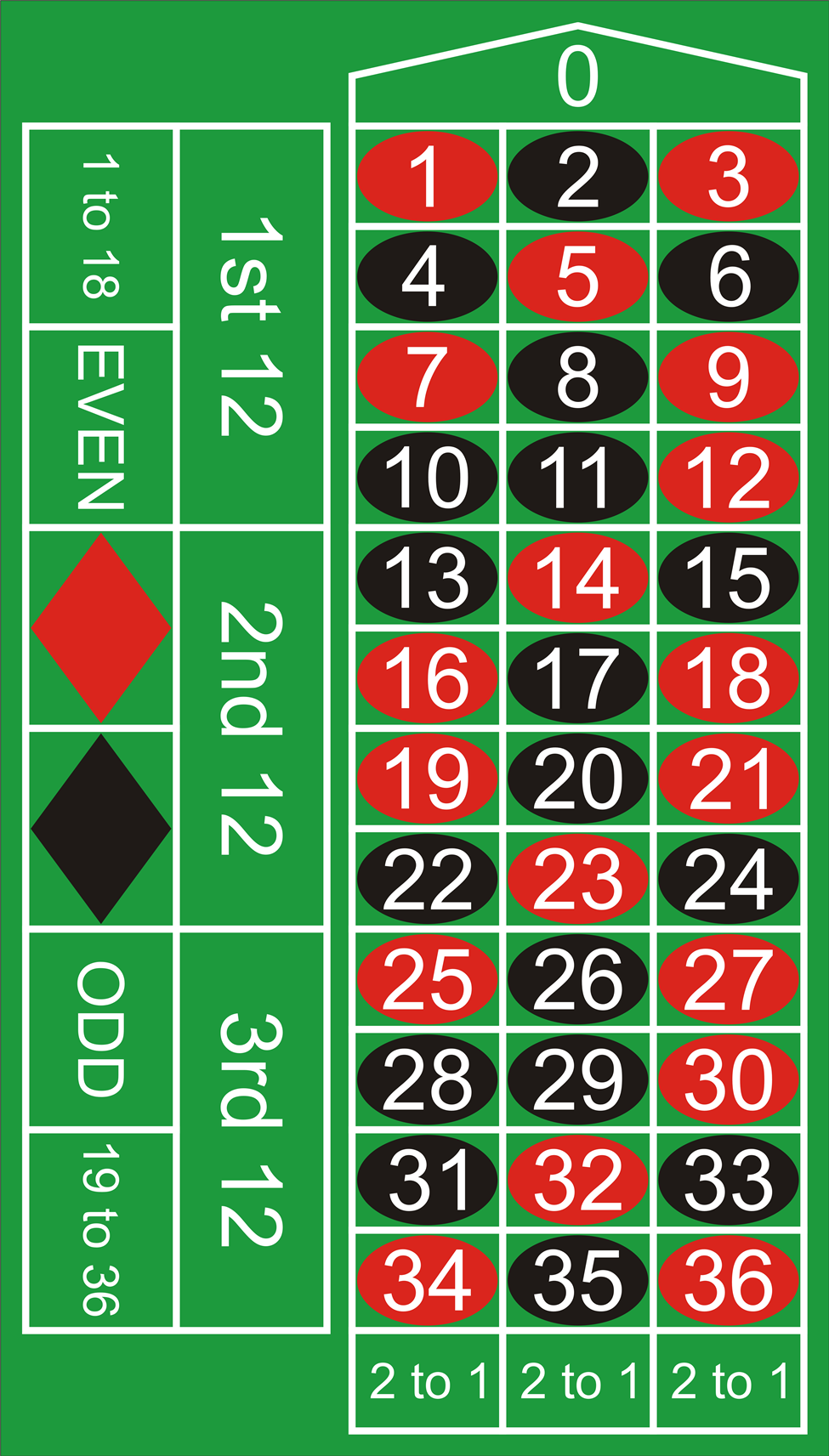}
\caption{Roulette seen as a multi-goal coupon collector's problem.}\label{fig:roulette}
\end{center}
\end{figure}

\section{It's all about counting transversals} \label{sec:counting}

In this and the next section all trials are silently assumed to be \emph{without} replacement. Mathematically our approach to the GCCP is straightforward; the challenge is its algorithmic realization in the next section. As to the mathematics, for fixed $k \in \{0,1,\cdots, w\}$ let $q_k$ be the probability that a length $k$ trial is successful. In particular $q_0 =0$ and $q_w=1$.

Recall that $G_i \subseteq W$ is the set of coupons of the $i$-th goal $(i\in [h])$. The {\it hypergraph} ($=$ set system) ${\mathcal G} = \{G_1, \dots, G_h\}$ fully determines all aspects of the GCCP. Specifically, $X \subseteq W$ is a {\it transversal} (or {\it hitting set}) of ${\mathcal G}$ if $X \cap G_i \neq \emptyset$ for all $i \in [h]$. Such a set $X$ of coupons displays each goal at least once, and so each permutation of $X$ corresponds to a successful trial. Conversely, each successful trial uses a set $X$ of coupons that is a transversal of ${\mathcal G}$.  Therefore, if
$$\tau_k \ := \textrm{number of $k$-element transversals of} \ {\mathcal G},$$
then exactly $k!\tau_k$ trials among the $w(w-1) \cdots (w-k+1)$ many length $k$ trials are successful, and so

\begin{equation}
q_k  =  \ds\frac{k!\tau_k}{w(w-1) \cdots (w-k+1)}   \left( =\frac{\tau_k}{{w\choose k}}\right) . \label{eq:qk} 
\end{equation}

Hence 
\begin{equation}
s_k  : =  q_k - q_{k-1} \qquad (k \in [w]) \label{eq:s}
\end{equation}
is the probability that a length $k$ trial is sharply successful. Therefore
\begin{equation}
\ell_{nr} ({\mathcal G})  =  \sum_{k=1}^w ks_k = \sum_{k=0}^{w-1} (1-q_k) = w - \sum_{k=1}^{w-1} q_k \label{eq:5}
\end{equation}
can be found by calculating the numbers $\tau_k$.

\section{The GCCP without replacement}% -- $\ell_{nr}({\mathcal G})$ as a fraction}

\label{sec:without-rep}

In this section all trials are still {\it without} replacement. Consider a set $W = \{c_1, \dots, c_8\}$ of eight coupons, each one of which serves between one and three goals, according to Table \ref{tab:1}.

For instance the trials $c_1, c_3, c_5$ and $c_6, c_2, c_8, c_7$ are successful. The first is sharply successful, the second is not. In order to calculate the expected length of a sharply successful trial, we  put ${\mathcal G}_1 := \{G_1, G_2, G_3, G_4\}$ and aim to count the $\tau_k$ many $k$-element transversals of ${\mathcal G}_1 \ (k \in [8])$. This is achieved by the transversal $e$-algorithm [Wild 2012a] which encodes a potentially gargantuan number of transversals of a hypergraph within comparatively few $\{0,1,2,e\}${\it-valued} rows. In the present article we are less concerned with {\it how} this method works (see \S\ref{sec:algorithm} for a few hints) but rather {\it what} it delivers. Namely, our {\it transversal hypergraph}
$${\Tr} ({\mathcal G}_1)   : =   \{X \subseteq W: \ X \ \mbox{is transversal of} \ {\mathcal G}_1 \}$$
gets presented as a union of five $\{0,1,2,e\}$-valued rows (Table \ref{Tab2}).

\begin{table}[h]
\begin{center}\begin{tabular}{c|c|c|c|c|c|c|c|c|}
& $c_1$ & $c_2$ & $c_3$ & $c_4$ & $c_5$ & $c_6$ & $c_7$ & $c_8$ \\ \hline
$G_1$ & $X$ & $X$ & $X$ & & & & & \\ \hline
$G_2$ & & &  $X$ & & & $X$ & $X$ & $X$ \\ \hline
$G_3$ & & $X$ & & $X$ & $X$ & $X$ & &  \\ \hline
$G_4$ & $X$ & & $X$ & $X$ & & $X$ & & $X$ \\ \hline
\end{tabular}\end{center}
\caption{Toy problem with 8 coupons and 4 goals}\label{tab:1}
\end{table}

\begin{table}[h]
\begin{center}\begin{tabular}{c|c|c|c|c|c|c|c|c|l}
& $c_1$ & $c_2$ & $c_3$ & $c_4$ & $c_5$ & $c_6$ & $c_7$ & $c_8$ & \\ \hline
$r_1$ & 2 & $e$ & 1 & $e$ & $e$ & $e$ & 2 & 2 & \quad $|r_1| = 120$ \\ \hline
$r_2$ & 1 & 0 & 0 & 2 & 2 & 1 & 2 & 2 & \quad $|r_2|  = 16$ \\ \hline
$r_3$ & 2 & 1 & 0 & 2 & 2 & $e$ & 2 & $e$ & \quad $|r_3| = 48$ \\ \hline
$r_4$ & $e$ & 1 & 0 & $e$ & 2 & 0 & 1 & 0 & \quad $|r_4| = 6$ \\ \hline
$r_5$ & 1 & 0 & 0 & $e$ & $e$ & 0 & $e'$ & $e'$ & \quad $|r_5| =9$ \\ \hline \end{tabular}\end{center}
\caption{${\Tr}({\mathcal G}_1)$ as disjoint union of $\{0,1,2,e\}$-valued rows}
\label{Tab2}
\end{table}

Each row encodes a set of $0,1$-bitstrings of length 8 that correspond to
subsets $X \subseteq W$ contained in ${\Tr}({\mathcal G}_1)$. Here 2 is a ``don't care''
symbol, i.e. the corresponding entry can be 0 or 1. Thus row $r_2$ encodes
$2^4$ bitstrings; one of them is $10010100$ which matches the transversal
$\{c_1, c_4, c_6\}$. A string of symbols $ee \cdots e$ (not necessarily on contiguous positions) by definition means
that any $0,1$-pattern with \emph{at least one 1} is allowed. In other
words, only $00 \cdots 0$ is forbidden. If several such $e$-patterns occur
within one row, they are mutually independent and notationally
distinguished. Thus row $r_5$ contains $(2^2 -1)(2^2-1) =9$ bitstrings, one
of them is $10011001$. 

By the workings of the $e$-algorithm the delivered $\{0,1,2,e\}$-valued rows are always mutually disjoint. For instance $r_3 \cap r_4= \emptyset$ because each $X \in r_3$ has $X \cap \{c_6, c_8\} \neq \emptyset$, and each $Y \in r_4$ has $Y \cap \{c_6, c_8\}  = \emptyset$. It follows that
$$
| {\Tr} ({\mathcal G}_1)|   =  120 + 16 + 48 + 6 + 9  =  199.
$$

For any $\{0,1,2,e\}$-valued row $r$ on a set $W$ and any $k \in [w]$ put

\begin{equation}
\Card(r,k)   : =   | \{X \in r: \ |X| = k\}|. \label{eq:6}
\end{equation}

If the set ${\Tr}({\mathcal G})$ of all transversals of some hypergraph ${\mathcal G}$ on $W$ is represented as disjoint union of $\{0,1,2,e\}$-valued rows $r_1, r_2, \dots, r_R$, then

\begin{equation}
\tau_k   =   \Card(r_1, k) + \ \Card(r_2, k) + \cdots + \ \Card(r_R, k). \label{eq:7}
\end{equation}

While there is a systematic way to get $\Card(r, k)$ (see \S\ref{sec:algorithm}), for ${\mathcal G} = {\mathcal G}_1$ we proceed ad hoc\footnote{For instance, the transversals counted by e.g.\ $\Card(r_5, 4) = 4$ are $\{c_1,c_4,c_5,c_7\}, \ \{c_1,c_4,c_5,c_8\}, \ \{c_1, c_4,c_7,c_8\}$, and $\{c_1,c_5,c_7,c_8\}$.} and get

$$
\begin{tabular}{c|c|c|c|c|c|c|c|c|c|c|c|c|c|}
$k=$& 1 & 2 & 3& 4 & 5& 6 &7 & 8 &   &  $|r_i|$ \\ \hline
$\Card(r_1, k)=$ & 0 & 4 & 18 & 34 & 35 & 21 & 7 &1 &  & 120 \\ \hline
$\Card(r_2, k)=$ & 0 &1 & 4 & 6 & 4 & 1 & 0 &0 &  & 16\\ \hline
$\Card(r_3, k)=$ & 0 &2 & 9 & 16 & 14 & 6 &1 & 0 & &  48 \\ \hline
$\Card(r_4, k)=$ & 0 &  0 & 2 & 3 & 1 & 0 & 0 & 0 &  & 6\\ \hline
$\Card(r_5,k)=$ &0 & 0 & 4 & 4 & 1 &0 & 0 & 0 &   &  9 \\ \hline 
 \hline 
$\tau_k=$ & 0 & 7 & 37 & 63 & 55 & 28 & 8 & 1 &  & \quad 199 \\
  \hline \end{tabular}
$$
\medskip

Having the $\tau_k$'s we can evaluate the probability $q_k$ of having a
successful trial of length $k$ by Formula \eqref{eq:qk} and get 
$$
\begin{array}{|c|c|c|c|c|c|c|c|} \hline
q_1 & q_2 & q_3 & q_4 & q_5 & q_6 & q_7 & q_8 \\ \hline
\rule{0pt}{12pt}  % make some extra space for the fractions
0 & \frac{1}{4} & \frac{37}{56} & \frac{9}{10} & \frac{55}{56} &  1 & 1 & 1 \\[2pt] \hline \end{array}
$$

Hence (\ref{eq:5}) gives
\begin{equation}
\ell_{nr} ({\mathcal G}_1)   =  8 - q_7 - \cdots - q_2 - q_1  =  \frac{449}{140}   \approx   3.2. \label{eq:8}
\end{equation}

\section{The GCCP with replacement}% -- $\ell_r({\mathcal G})$ as a priori infinite sum} 

\label{sec:with-rep}

Without further mention all trials in this section are \emph{with} replacement. Let $t'_n$ be the number of successful length $n$ trials, i.e. trials where all goals of coupons have occured at some point (so $t'_0 =0$). Thus
\begin{equation}
q'_n   : =  \ds\frac{t'_n}{w^n} \qquad (n \geq 0) \label{eq:9}
\end{equation}
is the probability that a length $n$ trial is successful, and
\begin{equation}
s'_n   : =   q'_n - q'_{n-1}  \qquad (n \geq 1) \label{eq:10}
\end{equation}
is the probability that a length $n$ trial is sharply successful. Hence the expected length of a sharply successful trial is
\begin{equation}
\ell_r ({\mathcal G})  =   \ds\sum_{n=1}^\infty n s'_n . \label{eq:11}
\end{equation}
As to calculating the numbers $t'_n$, observe that no matter how coupons $c_i$ are repeated in a length $n$ trial, the underlying set of (distinct)
coupons must be a $k$-element transversal $X$ of ${\mathcal G}$, for some $k \leq n$. For a fixed $k$-element set of coupons $X \subseteq W$
the number of length $n$ trials with underlying set $X$ equals the number of ways to distribute $n$ distinct balls (corresponding to the positions in the
trial) to $k$ distinct buckets (corresponding to the coupons) in such a way that no bucket stays empty. It is well known that this number is $k! S(n,k)$,
where $S(n,k)$ is the Stirling number of the second kind. Accordingly, with $\tau_k$ as in \S\ref{sec:counting}, we deduce that
\begin{equation}
t'_n = \left\{ \begin{array}{lll}
1!\,\tau_1 S(n,1) + 2!\,\tau_2 S(n,2) + \cdots + n!\,\tau_n  S(n,n), & n \leq w,\\
\\
1!\,\tau_1 S(n,1) + 2!\,\tau_2  S(n,2) + \cdots + w!\,\tau_w  S(n,w), & n > w. \end{array} \right. \label{eq:12}
\end{equation}

Fortunately the infinite sum in (\ref{eq:11}) can be evaluated as a finite sum. What's more, one can express it in terms of the probabilities $q_k$
that a length $k$-trial \emph{without} replacement is successful. Recall how $q_k$ is coupled to $\tau_k$ according to (\ref{eq:qk}). 

\begin{theorem}  \label{thm:main} 
For drawings with replacement the expected length of a sharply successful trial is
$$\ell_r({\mathcal G})   = w \sum_{k=0}^{w-1} \frac{1-q_k}{w-k}  = w \left(H(w) - \ds\sum_{k=1}^{w-1}
\frac{q_k}{w-k}\right)  = w H(w) - \ds\sum_{k=1}^{w-1}
\frac{\tau_k}{{w-1\choose k}}.$$ 
\end{theorem}

\begin{proof}
Drawing with replacement yields an infinite sequence of coupons. If we ignore all repetitions and only keep the \emph{new} coupons, i.e., the coupons that 
have not occured earlier in the sequence, we obtain a subsequence of distinct coupons. With probability 1, every coupon occurs eventually, so the
subsequence will contain all $w$ coupons, and by symmetry they can come in any order with the same probability $1/w!$; hence, the subsequence of new
coupons is the same as drawing without replacement. 

If we draw with replacement we stop when we first have attained all goals, i.e, when the trial is sharply successful. Since repeated coupons do not help (or hinder), it is clear that we will
stop when we get a new coupon. Moreover, by the argument above, the probability that we stop when we get the $k$-th new coupon is precisely the probability $s_k$ given in \eqref{eq:s} that a trial with $k$ drawings without replacement is sharply successful. Furthermore, the positions of the new coupons are, by symmetry, stochastically independent of the sequence of values of the new coupons. Hence, provided that we stop at the $k$-th new coupon, the expected number of coupons drawn equals the expected number $e_k$ of drawings required to get $k$ distinct coupons (ignoring their values), which is known to be $e_k = \sum_{i=1}^k \frac{w}{w+1-i}$.  (See e.g.\ Feller [1957, Example IX.3(d)] for this well-known fact;
the standard argument is that when we have got $j$ distinct coupons, the probability that the next coupon is new is $(w-j)/w$, and thus the expected waiting time 
for the next new coupon is $w/(w-j)$.) Recalling that $q_0 =0$ and $q_w =1$, we deduce:
\begin{align*}
\ell_r({\mathcal G}) &= \sum_{k=1}^w s_k e_k = \sum_{k=1}^w \left( (q_k - q_{k-1}) \sum_{i=1}^k \frac{w}{w+1-i}\right) \\
&= (q_1-q_0) \frac{w}{w} + (q_2 - q_1)\left( \frac{w}{w} + \frac{w}{w-1}\right) + \cdots + (q_w-q_{w-1})\left( \ds\frac{w}{w} + \frac{w}{w-1} + \cdots + \frac{w}{1} \right)\\
&= \frac{w}{w} (q_w-q_0) + \frac{w}{w-1} (q_w-q_1) + \cdots + \frac{w}{2} (q_w-q_{w-2}) + \frac{w}{1} (q_w-q_{w-1}) \\
&= wH(w) -w \sum_{k=1}^{w-1}\frac{q_k}{w-k} = w \sum_{k=0}^{w-1} \frac{1-q_k}{w-k}.
\end{align*}
By (\ref{eq:qk}) we have $\tau_k = q_k{w \choose k}$ from which the rightmost formula in the Theorem follows.
\end{proof}

For instance, for our example Theorem \ref{thm:main} yields $\ell_r({\mathcal G}_1) = \frac{59}{15} \approx 3.9$ as opposed to $\ell_{nr}({\mathcal G}_1) \approx 3.2$ from (\ref{eq:8}). Notice that $\ell_r({\mathcal G}) = wH(w)$ if and only if all $q_k=0 \ (k < w)$, which is the classical case where each coupon has only one goal (and all these goals are distinct). The other extreme $\ell_r({\mathcal G}) = w \frac{1}{w}=1$ occurs if and only if all $q_k =1 \ (1\leq k \leq w)$, which means that every coupon fulfils every goal.

\begin{remark}
{\rm The key tool of our proof of Theorem~\ref{thm:main} is the fact that $q_k - q_{k-1}$ (the probability of a length $k$ trial without replacement being sharply successful) is also the probability that a trial with replacement is sharply successful when the $k$-th distinct coupon is drawn. This fact can also be used to compute the variance (and in principle also higher moments) of the trial length: to this end, note that the expectation of the square of the number of coupons needed to collect $k$ distinct coupons is (by the same argument as before, decomposing into $k$ independent geometrically distributed random variables)
$$\sum_{i=1}^k \frac{(i-1)w}{(w+1-i)^2} + \Bigg( \sum_{i=1}^k \frac{w}{w+1-i} \Bigg)^2.$$
Now repeating the argument of the  proof of Theorem~\ref{thm:main} yields the following expression for the variance:
$$\sum_{k=0}^{w-1} (1-q_k) \left( \frac{w(w+k)}{(w-k)^2} + \frac{2w^2}{w-k} (H(w) - H(w-k)) \right) - \ell_r({\mathcal G})^2.$$
In our standard example, this yields a variance of $\frac{836}{225} \approx 3.7$. For drawings without replacement, the situation is much simpler, and the variance is
$$\sum_{k=0}^{w-1} (2k+1)(1-q_k) - \ell_{nr}({\mathcal G})^2,$$
which equals $\frac{18339}{19600} \approx 0.9$ in our example. This ends Remark 1.}
\end{remark}

Here comes a problem which carries over from coupons to goals in a more direct fashion. Referring to Table \ref{tab:1}, the probability that goal 1 does
\emph{not} belong to a randomly drawn coupon is $a_1 = \frac{5}{8}$. Similarly define $a_2, a_3, a_4$, so $a_i=1-m_i/w$ if goal $i$
is served by $m_i$ coupons.

Coupled to a random drawing of coupons, set the random variable $X_i: = 1$ if goal $i$ comes up, and $X_i : =0$ otherwise. Hence the expected value of $X_i$ in a length $n$
trial is $E_n [X_i] = 1-a^n_i$. For drawing without replacement, the corresponding formula is
$$
E_n[X_i] = 1- \frac{\binom{w-m_i}{n}}{\binom{w}{n}}= 1- \frac{(w-m_i)!\, (w-n)!}{(w-n-m_i)!\, w!},
$$
interpreted as $1$ if $m_i+n > w.$

By the linearity of expectation one calculates that $e_n = \sum_{i=1}^4E_n[X_i]$ is the expected number of goals gathered in a
length $n$ trial. For instance, $e_4 \approx 3.7$ for drawing with replacement. 

\section{ The transversal $e$-algorithm: past and future}\label{sec:algorithm}

Independent of coupons, let ${\mathcal G}$ be any hypergraph based on $W$. One calls $X \subseteq W$ a {\it noncover} of ${\mathcal G}$ if $X \not\supseteq G$ for all $G \in{\mathcal G}$.  The ({\it noncover}) $n$-{\it algorithm} generates, compactly encoded as certain $\{0,1,2,n\}$-valued rows, all noncovers of ${\mathcal G}$. Its soundness and output-linear complexity was established in [Wild 2012a]. Note that $X$ is a noncover of ${\mathcal G}$ if and only if $W \setminus X$ is a transversal of ${\mathcal G}$. Hence one can get all transversals by running the $n$-algorithm. But for succinctness it pays to dualize everything from scratch and call the result the ({\it transversal}) $e$-{\it algorithm}. In particular, whereas $nn \dots n$ means ``at least one $0$ here'', the wildcard $ee \dots e$ means ``at least one 1 here''.

This duality was pointed out already in [Wild 2012a]. We mention that akin to the present article, also [Wild 2012b] combines the $n$-algorithm with certain probabilities, but in a completely different way. 

In the remainder of \S\ref{sec:algorithm} we generalize transversals to transversouls (to prepare for \S\ref{sec:alpha}) and will encode the latter by a suitable 
adaption of $\{0,1,2,e\}$-valued rows. The underlying transversoul $e$-algorithm will however be discussed in a future publication.  

\bigskip

Let ${\mathcal G} = \{G_1, \dots, G_h \}$ be a set system on $W = [w]$ and let $\alpha = (\alpha_1, \dots, \alpha_h)$ be a fixed vector with integer components $\alpha_i  \geq 1$. If $X \subseteq W$ is such that

\begin{equation}
(\forall 1 \leq i \leq h) \quad |X \cap G_i| \geq \alpha_i \label{eq:13}
\end{equation}
then $X$ is called, tongue in cheek,  an \emph{$\alpha$-transversoul} because it has more ``soul'' than an ordinary transversal where all $\alpha_i =1$.

As an example, take $W = [12]$ and let
$$\begin{array}{lll}
{\mathcal G}_2 & : = & \{ \{1,2,3,4,5\}, \{6,7,8\}, \{4,5,6,9,10, 11, 12\}\}, \\
\alpha & : = & (\alpha_1, \alpha_2,\alpha_3)  =  (2,1,3). \end{array}$$

Akin to \S\ref{sec:without-rep} we wish to compactly encode the family ${\Tr}({\mathcal G}_2, \alpha)$  of all $\alpha$-transversouls as a disjoint
union of multivalued rows. The key idea is the symbolism $e(s) e(s) \dots e(s)$, where $s \geq 1$ is an integer, and where the number $m$ of symbols
$e(s)$ occuring must be greater than $s$. By definition only $0,1$-bitstrings $X$ are allowed that have \emph{at least $s$ entries $1$}
on the positions occupied by the symbols $e(s)$. In particular (say) $e(1) e(1) e(1)$ amounts to the previously introduced $eee$. 

It is clear that if ${\mathcal G}_2$ was $\{G_1, G_2\}$ then the family of all $(2,1)$-transversouls $X \subseteq W$ could be written as the row $\ol{r}_0$
in Table~\ref{tab:4}. (For systematic reasons we always write $e(1) \dots e(1)$ rather than $e \dots e$.) It is more demanding to find a neat
representation of the subset ${\Tr} ({\mathcal G}_2, \alpha)$ of $\ol{r}_0$. This brings about the use of another wildcard $g(s) g(s) \dots g(s)$ (more than $s$ symbols $g(s)$) which by definition means that only bitstrings are allowed that have {\it exactly $s$ entries} 1 on the positions occupied by the symbols $g(s)$. It turns out\footnote{As said, a systematic treatment will appear elsewhere. This footnote only gives a few hints. Focus on $G_1 \cap G_3 = \{4,5\}$ and $G_2 \cap G_3 = \{6\}$. Combining the three options $00, \ g(1) g(1), \ 11$ of the former
intersection with the two options $0,1$ of the latter intersection (boldface entries in Table~\ref{tab:4}) results   in the six mutually disjoint rows $\ol{r}_1$ to $\ol{r}_6$. Consider   e.g.\ row $\ol{r}_4$. Switching the last two components of $e(2)e(2)e(2)e(2)e(2)$ in $\ol{r}_0$ to $g(1)g(1)$ in $\ol{r}_4$ forces its first three components to be $e(1) e(1) e(1)$ in $\ol{r}_4$. Similarly, switching the first component of $e(1)e(1)e(1)$ in $\ol{r}_0$ to $1$ in $\ol{r}_4$ frees its last two components to be $22$ in $\ol{r}_4$.  What's more, since $g(1)g(1)1$ in $\ol{r}_4$ means that
either $\{4,6\} \subseteq X$ or $\{5,6\} \subseteq X$ for all $X \in \ol{r}_4$, one needs to write $e'(1) e'(1) e'(1)e'(1)$ at the end of $\ol{r}_4$ in order to ensure that $|X \cap G_3| \geq 3$ for all $X \in \ol{r}_4$.} that ${\Tr}({\mathcal G}_2, \alpha)$ is the disjoint union of the $\{0,1,2,e(s), g(s) \}$-valued rows $\ol{r}_1, \ol{r}_2, \dots, \ol{r}_6$ in Table \ref{tab:4}. It is an easy matter to calculate the cardinality of a $\{0,1,2, e(s), g(s)\}$-valued row; say 
$$|\ol{r}_3|  =   (2^3-1) \cdot 2 \cdot (2^2-1) \cdot (2^4-5)  =  462.$$

\begin{table}[h]
\begin{center}\begin{tabular}{l|c|c|c|c|c|c|c|c|c|c|c|c|c|c|c|c|}
& 1 & 2 & 3 & 4 & 5 & 6 &7 & 8 & 9 & 10 & 11 & 12 \\ \hline
$\ol{r}_0=$ & $e(2)$ & $e(2)$ & $e(2)$  & $e(2)$ & $e(2)$ & $e(1)$ & $e(1)$ & $e(1)$ & 2 & 2 & 2 & 2\\ \hline
%& & & & & & & & & & & & \\ \hline
$\ol{r}_1=$ & $e(2)$ & $e(2)$ & $e(2)$ & ${\bf 0}$ & ${\bf 0}$ & ${\bf 0}$ & $e(1)$ & $e(1)$ & $e(3)$ & $e(3)$ & $e(3)$ & $e(3)$\\ \hline
$\ol{r}_2=$ & $e(2)$ & $e(2)$ & $e(2)$ & ${\bf 0}$ & ${\bf 0}$ & ${\bf 1}$ & 2 & 2& $e'(2)$ & $e'(2)$ & $e'(2)$ & $e'(2)$\\ \hline
$\ol{r}_3=$ & $e(1)$ & $e(1)$ & $e(1)$ & ${\bf g(1)}$ & ${\bf g(1)}$ & ${\bf 0}$ & $e'(1)$ & $e'(1)$ & $e(2)$ & $e(2)$ & $e(2)$ & $e(2)$ \\ \hline
$\ol{r}_4=$ & $e(1)$ & $e(1)$ & $e(1)$ & ${\bf g(1)}$ & ${\bf g(1)}$ & ${\bf 1}$ & 2 & 2 &  $e'(1)$ & $e'(1)$ & $e'(1)$ & $e'(1)$ \\ \hline
$\ol{r}_5=$ & 2 & 2 & 2  & ${\bf 1}$ & ${\bf 1}$ & ${\bf 0}$ & $e(1)$ & $e(1)$ & $e'(1)$ & $e'(1)$ & $e'(1)$ & $e'(1)$ \\ \hline
$\ol{r}_6=$ & 2 & 2 & 2 & ${\bf 1}$ & ${\bf 1}$ & ${\bf 1}$ & 2 & 2 & 2 & 2 & 2 & 2 \\ \hline
\end{tabular}\end{center}

\caption{The hypergraph ${\Tr}({\mathcal G}_2, \alpha)$ as a disjoint union of $\{0,1,2,e(s), g(s)\}$-valued rows}
\label{tab:4}
\end{table}

The numbers $\Card(r,k)$ (defined as in (\ref{eq:6})) are obtained by noting that $\Card(r,k)$ is the coefficient at $x^k$ of the polynomial $\pol(r,x)$ which is defined as the product of these factors: For each symbol 1 in $r$ take a factor $x$, for each symbol 2 a factor $1+x$, for each constraint $e(s) \dots e(s)$ of length $m$ a factor ${m \choose s} x^s + {m \choose s+1} x^{s+1} + \cdots + {m \choose m} x^m$, and for each constraint $g(s) \dots g(s)$ of length $m$ a factor ${m \choose s} x^s$. Thus, for instance
$$\begin{array}{lll}
\pol(\ol{r}_3, x) & =& (3x+ 3x^2+x^3) (2x)(2x+x^2)(6x^2+4x^3+x^4)\\
\\
& =& 72x^5+156 x^6+144x^7+70x^8+18x^9 +2x^{10} \end{array}$$
and so e.g.\ $\Card(\ol{r}_3, 7) = 144$.

If $T_k$ is defined as the number of $k$-element $(2,1,3)$-transversouls of
${\mathcal G}_2$ then similar to \eqref{eq:7} we have 
\begin{equation}\label{eq:14}
T_k = \ \Card(\ol{r}_1, k) + \cdots + \ \Card (\ol{r}_6, k).
\end{equation}

Apart from the previous lengthy footnote, here comes another advertisement of the transversoul $e$-algorithm. Namely, each constraint $e(s) \dots e(s)$ of length $m$ can be written as an intersection of $\ds{m \choose m-s+1}$ many old type $e$-constraints, for instance
$$(e(2), e(2), e(2), e(2))  =  (e,e,e,2) \cap (e,e,2,e) \cap (e,2,e,e) \cap (2,e,e,e).$$
Hence besides the practical benefits also the theoretic properties of the $e$-algorithm will carry over.

The remaining three sections are dedicated to numerical evaluations and applications of the $e$-algorithm, with a pinch of transversouls in the last section.

\section{The nonhomogenous CCP: Pitting the $e$-algorithm against\\ inclusion-exclusion}\label{sec:comparison}

Boneh and Hofri [1997, p.~43] emphasize the computational difficulty to evaluate (\ref{eq:1}) as $h$ increases, and then go on to use integration for approximation. Recall that for rational $p_i$'s in the (classic) CCP, say
$$p_1 = \frac{1}{10}, \ p_2 = \frac{2}{10}, \ p_3 = \frac{3}{10}, \ p_4 = \frac{4}{10},$$
our approach uses $W = [10]$ and the partition
$${\mathcal G}^\ast   =  \{\{1\}, \{2,3\}, \{4,5,6\}, \{7, 8, 9, 10\} \}.$$
Because the sets in ${\mathcal G}^\ast$ are disjoint, we can do with a \emph{single} $\{0,1,2,e\}$-valued row
$$r  =   (1, \ \ e_2, e_2, \ \ e_3, e_3, e_3, \ \ e_4, e_4, e_4, e_4).$$
Using $\pol(r,x)$ from \S\ref{sec:algorithm} one readily computes the numbers $\tau_k = \ \Card(r,k)  \ (k \in [10])$, and from them $\ell_r({\mathcal G}^\ast)$
according to Theorem \ref{thm:main}. 

\begin{table}[h]
\begin{center}\begin{tabular}{c|c|c|c|}
$h$ & $\ell_r({\mathcal G}^\ast)$ & exclusion  & incl-excl.  \\ \hline
$10$ & $68.9846$ & $0$ & $0.2$\\ \hline
15 & $150.606$ & $0$ & $7.7$ \\ \hline
27 & $474.463$ & $0.3$ & $43193$\\ \hline
50 & $1600.38$ & $4.1$ & -\\ \hline
100 & $6338.75$ & $72$ & - \\ \hline
150 & $14215.1$ & $455$ & - \\ \hline
200 & $25229.5$ & $1829$ & - \\ \hline
400 & $100667$ & $96272$  & - \\ \hline
\end{tabular}\end{center}

\caption{Total time in seconds taken when computing $\ell_r(\mathcal G^\ast)$
by the $e$-algorithm (exclusion) and by the inclusion-exclusion algorithm.}
\label{tab:3}
\end{table}

Table \ref{tab:3} compares the $e$-algorithm with the inclusion-exclusion approach (\ref{eq:1}) on instances $(p_1, \dots, p_h)$ of the particular but natural type 
$$p_1 = \frac{1}{w}, \quad p_2 = \frac{2}{w}, \quad \dots, \quad p_h = \frac{h}{w} \quad \left(\mbox{hence} \ w = 1+ \cdots + h= \ds\frac{h(h+1)}{2} \right)$$
which is uniquely defined by $h$ ($=$ first column in Table~\ref{tab:3}). As to inclusion-exclusion, we used a standard Gray-code in order to more economically generate the subsets of $[h]$ one by one from their predecessors, and also used that for common denominator probabilities one can simplify the terms in (\ref{eq:1}); say
$$\frac{1}{p_i+p_j+p_k} = \frac{1}{\frac{i}{w} + \frac{j}{w} + \frac{k}{w}} = \frac{w}{i+j+k}.$$
The value of $\ell_r({\mathcal G}^\ast)$ is rounded to 6 digits albeit Mathematica, provided with the numbers $\tau_k  \ (k \in [w])$, delivered 
the \emph{exact} value as a fraction of two very large integers. For instance $h = 400$ gives $w = 80200$ and $3108$ sec of the $96272$ sec total
time were spent on plugging $\tau_1, \tau_2, \dots, \tau_{80200}$ into the formula of Theorem \ref{thm:main}. As is apparent, inclusion-exclusion
(formula (\ref{eq:1})) cannot compete. 

For the particular $p_i$'s considered one can show [David and Barton 1962, p.269] that $\ell_r ({\mathcal G}^\ast)$ asymptotically goes to $\left( \frac{4\pi}{\sqrt{3}} - 6 \right) {h+1 \choose 2}$ as $h \ra \infty$. Already for $h=15$ the latter gives the tight approximation $150.624$ to the true (rounded) value $150.606$.

%\begin{tabular}{c|c|c|c|c|c|c} \hline

%$w$ & $h$ & $d$ & $\ell_{nr}({\mathcal G})$ &  $\ell_r({\mathcal G})$ &  time (sec) & final rows \\ \hline

%40  & 10  & 0.3 & 7.69 & 8.57 &  1.6 & 579 \\ \hline

%40  & 10 & 0.7 & 2.78 & 2.86 &  0.3 & 123 \\ \hline

%40 & 40 & 0.3 & 10.61 & 12.28 &  686.3 & 274 117\\ \hline

%40 & 40 & 0.7 & 3.85 & 4.00 &  6.0 & 2092 \\ \hline

%40 & 100 & 0.3 & 12.34 & 14.66 &  18349.4 & 5 429 569\\ \hline

%40 & 100 & 0.7 & 4.46 & 4.68 &  44.5 & 9842 \\ \hline

%100 & 100 & 0.7 & 4.66 & 4.75 &  2435.5  & 361428 \\ \hline

%\end{tabular}

%Table 5

\section{Information spreading and the expected time to dominate a chess board}  \label{sec:domination}

In many GCCP applications the goals of a coupon $c$ are \emph{other coupons}, namely those that $c$ wishes to ``influence'' in some way. More succinctly, we may consider a graph $G$ whose vertex set $W$ we imagine as a group of $h=w$ people whose friendship relations are reflected by the edges of $G$. Suppose members $c \in W$ are phoned at random from outside $W$ and told a piece of information. If $c$ shares the news with all his friends, what is the expected\footnote{We mention in passing that the \emph{minimum} number of phone calls necessary is called the \emph{domination number} of $G$.} number $\ell_r({\mathcal G})$ of phone calls necessary before the whole of $W$ is informed? What is the analogue number $\ell_{nr}({\mathcal G})$ when nobody is phoned twice?

A pleasant instance of the graph framework, where friendship turns to aggression though, is the problem to determine the expected number $\ell_{nr}$(queens) of queens it takes when they are placed on a chessboard at random until the queens dominate the board, i.e., all 64 squares (coupons) are occupied or threatened. If occupied squares can still be drawn (without effect apart from increasing the trial's length), let $\ell_r$(queens) be the corresponding number. We also define $\ell_{nr}$(rooks), $\ell_{nr}$(kings), \ldots in an analogous fashion.

One obtains the following results (rounded to four decimals):

$$\begin{array}{llllll}
\ell_{nr} ({\rm queens}) &= & 11.8402 & \quad \ell_r ({\rm queens}) & =& 15.2945\\
\ell_{nr} ({\rm rooks}) &=  & 15.0045 & \quad \ell_r ({\rm rooks}) &= & 17.1308\\
\ell_{nr} ({\rm kings}) & = & 30.4091 & \quad \ell_r ({\rm kings}) & =& 42.4282 \end{array}$$
If one does not consider a square occupied by a queen as threatened by her (after all, an unthreatened knight can capture her), the numbers $\ell_{nr}$(queens) and $\ell_r$(queens) grow to $\ell^\ast_{nr}$(queens)$= 12.7094$ respectively $\ell^\ast_r$(queens) $= 16.3149$. Also GCCP applications to trading card games such as {\it Magic: The Gathering}, and much more\footnote{Readers are encouraged to email ideas to any of the authors.}, are conceivable.

%\subsection{From transversals to transversouls}  \label{sec:souls}

%\subsection{The transversoul $e$-algorithm in a nutshell}  \label{subsec:nutshell}

%\subsection{Focusing on individual rows}  \label{subsec:focusing}

\section{The likelihood of getting the $i$-th goal at least $\alpha_i$ times in $k$ drawings}  \label{sec:alpha}

For the homogeneous CCP with $w$ coupons the probability $p$ to have exacty $k$ distinct coupons after a length $n$ trial was already known to Laplace,
and is easily seen to be 
$$
p   =   \ds\frac{k!}{w^n} {w \choose k} S(n,k).
$$

For the nonhomogeneous CCP matters get more complicated. Settling for ``at least $k$'' instead of ``exactly $k$'', the problem has e.g.\ been tackled in
[Boneh-Hofri 1997] by using the Cauchy integral formula. 

Here we also stick to $\geq k$ but lift the problem from coupons to goals and focus on trials \emph{without} replacement. No surprise, our approach is completely different.
Consider a set $W = \{c_1, c_2, \dots, c_{12} \}$ of twelve coupons, each one of which having one or two goals as in Table~\ref{tab:6}.

\begin{table}[h]
\begin{center}\begin{tabular}{|l|c|c|c|c|c|c|c|c|c|c|c|c|}
& $c_1$ & $c_2$ & $c_3$ & $c_4$ & $c_5$ & $c_6$ & $c_7$ & $c_8$ & $c_9$ & $c_{10}$  & $c_{11}$ & $c_{12}$ \\ \hline
goal 1 & $x$ & $x$ & $x$ & $x$ & $x$ & & & & & & &\\ \hline
goal 2 & & & & & & $x$ & $x$ & $x$ & & & & \\ \hline
goal 3 & & & & $x$ & $x$ & $x$ & &  & $x$ & $x$ & $x$ & $x$\\ \hline
\end{tabular}\end{center}

\caption{Toy problem for \S\ref{sec:alpha}}
\label{tab:6}
\end{table}

Fix $k \in [12]$. If the coupons are drawn with equal probability $\frac{1}{12}$ and without replacement, what is the likelihood $Q_k$ that
after exactly $k$ drawings the goals $1,2,3$ have shown up at least $2, 1, 3$ times respectively? One readily verifies that
$$Q_1 = Q_2 = 0, \quad Q_3 > 0, \quad Q_9 < 1, \quad Q_{10} = Q_{11} = Q_{12} = 1.$$
As to why $Q_3 > 0$, look at $c_4, c_5, c_6$.
The precise values of the $Q_k$'s are as in \eqref{eq:qk} given by (here $w =12$) 
\begin{equation}
Q_k  =  \ds\frac{k!T_k}{w(w-1) \cdots (w-k+1)} \label{eq:16}
\end{equation}
where $T_k$ is the number of $k$-element $(2,1,3)$-transversouls of the hypergraph based on $\{c_1, \dots, c_{12} \}$ which is induced by Table~\ref{tab:6}. It just so happens that this is the hypergraph  ${\Tr} ({\mathcal G}_2, \alpha)$ from \S\ref{sec:algorithm}. The $i$-th row in Table \ref{tab:8} contains the
numbers $\Card(\ol{r}_i, 1)$ up to $\Card(\ol{r}_i, 12)$, which were calculated in the same way as $\Card(\ol{r}_3, 7) = 144$ in \S\ref{sec:algorithm}. Hence $T_k$ is the sum
of the numbers of the $k$-th column in Table \ref{tab:8}. The $T_k$'s yield the desired probabilities $Q_k$ according to \eqref{eq:16}. 

Finally, consider the generalization of GCCP where the drawing probabilities of coupons {\it depend} on the previously drawn coupons. That the $e$-algorithm adapts to some extent, will (hopefully) be shown in a future publication.

\begin{table}
\begin{center}\begin{tabular}{|l|c|c|c|c|c|c|c|c|c|c|c|c|}
& 1 & 2& 3& 4& 5 & 6 & 7 & 8 & 9 & 10 & 11 & 12 \\ \hline
$\ol{r}_1$ & 0 & 0 & 0 & 0 & 0 & 24  & 26 & 9 & 1 & 0 &0 &0 \\ \hline
$\ol{r}_2$ & 0 & 0 & 0 & 0 & 18 & 54 & 61 & 33 & 9 &1 &0 & 0 \\ \hline
$\ol{r}_3$ & 0 & 0 & 0 & 0 & 72 & 156 &144 & 70 & 18 & 2 & 0 & 0 \\ \hline
$\ol{r}_4$ & 0 & 0 & 0 & 24 & 108 & 212 & 238 & 166 & 72 & 18 & 2 & 0 \\ \hline
$\ol{r}_5$ & 0 & 0 & 0 & 8  & 40 & 86 & 104 & 77 & 35 & 9 & 1 &  0 \\ \hline
$\ol{r}_6$ & 0 & 0 & 1 & 9 & 36 & 84 & 126 & 126 & 84 & 36 & 9 & 1 \\ \hline
$T_k$ & 0 & 0 & 1 & 41 & 274 & 616 & 699 & 481 & 219 & 66 & 12 & 1 \\ \hline
$Q_k$ & 0 &0 & 0.005 & 0.083 & 0.346 & 0.667 & 0.883 &0.972 &  0.995 & 1 & 1 & 1 \\ \hline
\end{tabular}\end{center}

\caption{Solution of toy problem of \S\ref{sec:alpha}}
\label{tab:8}
\end{table}

%\nocite{*}

%\bibliographystyle{abbrvnat}

% use the following instead if you encounter problems

%\bibliographystyle{alpha}

%\bibliography{coupons}

%\label{sec:biblio}

\section*{References}

\begin{description}

\item [Boneh--Hofri 1997] A. Boneh and M. Hofri, The Coupon-Collector
	  Problem revisited --  a survey of engineering problems and computational
	  methods. {\it Stochastic Models} 13 (1997), 39--66.
\item [David-Barton 1962] F.N. David and D.E. Barton, Combinatorial Chance, Charles Griffin \& Company Limited 1962.

\item[Feller 1957] W. Feller, {\it An introduction to probability theory
	  and its applications, Vol. I.} Wiley, New York, 2nd edition, 1957.

%    \item[FGT 1992] P. Flajolet, D. Gardy, L. Thimonier, Birthday paradox,

%    coupon collectors, caching algorithms and self-organizing search,

%    Disc. Appl. Math. 39 (1992) 207-229. 

    \item[Wild 2012a] M. Wild, Compactly generating all satisfying truth
	  assignments of a Horn Formula. 
\emph{Journal on Satisfiability, Boolean Modeling and Computation} 8 (2012), 63--82.

    \item[Wild 2012b] M. Wild, Computing the output distribution and
	  selection probabilities of a stack filter from the DNF of its positive
	  Boolean function. 
\emph{Journal of Mathematical Imaging and Vision}, Online First, 1 August 2012.

%    \item[Knuth 1998]    D. E. Knuth, 

% \emph{The Art of Computer Programming. 

% Vol. 3: Seminumerical Algorithms}.

% 3nd ed., Addison-Wesley,

% Reading, Mass., 1998.

% \item[Dawkins 1991] Brian Dawkins, 

% "Siobhan's problem: the coupon collector revisited", 

% The American Statistician 45 (1): 76–82, 1991.

% \item[BHS 1994] G. Blom, L. Holst and D. Sandell,

% Problems and snapshots from the world of probability.  Springer-Verlag, New

% York, 1994. xii+240 pp. ISBN: 0-387-94161-4 

% \item[Polya 1930] G. P\'olya,

% Eine Wahrscheinlichkeitsaufgabe in der Kundenwerbung 

% Zeitschrift f\"ur Angewandte Mathematik und Mechanik

% 10 (1930), 96--97

% \item[Motwani-Raghavan 1995] R. Motwani and P. Raghavan, 

% Randomized algorithms.  Cambridge University Press, Cambridge, 1995. 

% xiv+476 pp. ISBN: 0-521-47465-5

\end{description}

\end{document}